\documentclass[12pt]{amsart}
\usepackage{amsmath,amsthm,amsrefs}
\usepackage{amsbsy,amsfonts,amssymb,mathtools}
\usepackage{amsaddr, ytableau,mathtools}
\def\sc{\scriptstyle}
\def\ds{\displaystyle}
\def\cc{\mbox{cc}}
\def\rw{\mbox{rw}}

\newtheorem{thm}{Theorem}[section]
\newtheorem{lem}[thm]{Lemma}

\newtheorem{cor}[thm]{Corollary}

\theoremstyle{definition}
\newtheorem{exa}[thm]{Example}

\title[Cyclic sieving for a family of semistandard tableaux]%
{Cyclic sieving for a family of semistandard tableaux}

\author{Joshua Basman Monterrubio, Graeme Henrickson \ \ \ \ \ \ \ \ \ \ \ \ \ \ \ \  and
Anna Stokke}
\address{University of Winnipeg \\
Department of Mathematics and Statistics \\
Winnipeg, Manitoba \\
Canada  R3B 2E9}

\email{\tt a.stokke@uwinnipeg.ca}
\thanks{This research was supported by NSERC grant RGPIN-2018-05877.}

\begin{document}

\begin{abstract}
We give a new cyclic sieving phenomenon for semistandard Young tableaux $SSYT(\lambda,\mu)$ of shape $\lambda=(m,n^b)$ and content $\mu$, a $(b+2)$-tuple.  We prove that $(SSYT(\lambda,\mu),\langle \partial^{b+2} \rangle, f(q))$ exhibits the cyclic sieving phenomenon, where $\partial$ is the jeu de taquin promotion operator and $f(q)$ is a modified Kostka-Foulkes polynomial $\widetilde{K}_{\lambda,\mu}(q)$, up to a power of $q$. 

 \end{abstract}

\keywords{Young tableau, jeu-de-taquin, cyclic sieving phenomenon}

\maketitle

\section{Introduction}

Given a finite set $X$ and a cyclic group $\langle g \rangle$ of order $n$ that acts on $X$, we can consider the cardinality of the fixed point set $X^{g^d}$, for a positive integer $d$.    The triple $(X, \langle g \rangle,f(q))$, where $f(q) \in \mathbb{N}[q]$, is said to exhibit the cyclic sieving phenomenon (CSP) if $\vert X^{g^d} \vert =f(\omega^d)$ for all $d\geq 0$, where $\omega$ is a primitive $n$th root of unity.  The cyclic sieving phenomenon was introduced by Reiner, Stanton and White in 2004 \cite{rsw} and has been widely studied since then, in various settings (see \cite{sagan} for details).

Several authors have produced CSPs for various sets of Young tableaux (see, for instance, \cite{linusson, bms, fontaine, gaetz, ohpark1, ohpark2, pechenik, psv, rhoades}).  Candidates for cyclic sieving polynomials are generally $q$-analogues of a natural counting formula (for example, the hook-length formula for standard tableaux) and a cyclic action on standard or semistandard tableaux is given by Sch\"utzenberger's jeu-de-taquin promotion operator $\partial$ (\cite{schutz1,schutz2}).  One roadblock is that the order of promotion (the least positive integer that fixes all tableaux in the set under $\partial$) is unknown for most shapes.  There are also situations where the order of promotion is known but the most natural cyclic sieving polynomial does not yield a CSP.  For example, the order of promotion for staircase tableaux was given in \cite{ponwang} but, so far, a CSP for staircase tableaux remains elusive.  

For standard rectangular tableaux of shape $\lambda=(a^b)$, the order of promotion is $ab$ \cite{haiman} and if $X=SSYT(a^b,k)$ is the set of semistandard rectangular tableaux with entries less than or equal to $k$, the order of promotion on $X$ is $k$.   Rhoades proved CSPs  for both standard and semistandard rectangular tableaux \cite{rhoades}.  We give a summary of CSP results for semistandard tableaux thus far.  For a list of CSPs in other settings, see \cite[Table 1]{linusson}. 

\bigskip

\noindent (1) Rhoades \cite{rhoades} proved that $(SSYT(\lambda,k),\langle \partial \rangle, q^{-\kappa(\lambda)}s_{\lambda}(1,q,\ldots,q^{k-1}))$ is a CSP triple, where  $\lambda$ is a rectangular partition, $s_{\lambda}(1,q,\ldots,q^{k-1})$ is a principal specialization of the Schur polynomial and $\kappa(\lambda)=\sum_i(i-1)\lambda_i$. 
\bigskip

\noindent (2) In \cite{fontaine}, the authors showed that $(SSYT(a^b,\gamma), \langle \partial^d \rangle, q^*K_{a^b,\gamma}(q))$ is a CSP triple, refining Rhoades's result.  Here $SSYT(a^b,\gamma)$ is the set of rectangular tableaux with fixed content $\gamma$, with $\gamma$ invariant under the $d$th cyclic shift, where $d$ is the frequency of $\gamma$---the number of cyclic shifts to return $\gamma$ to itself---and $q^*K_{a^b,\gamma}(q)$ is a Kostka-Foulkes polynomial up to a power of $q$.

\bigskip

\noindent (3) A CSP for semistandard hook tableaux  with content $\mu$  was given in \cite{bms} where it is shown that $(SSYT((n-m,1^m),\mu), \langle \partial^d \rangle, f(q))$
 is a CSP triple, with cyclic sieving polynomial $f(q)=\left[ \begin{array}{c}nz(\mu)-1 \cr  m \cr \end{array}\right]_q.$  Here $z(\mu)$ is the number of non-zero entries in $\mu$.

\bigskip

\noindent (4) Using the cyclic action $\tt{c}$ arising from the $U_q(\mathfrak{sl}_n)$ crystal structure for semistandard tableaux, Oh and Park \cite{ohpark1} proved
$(SSYT(\lambda),\langle \mathtt{c} \rangle, q^{-\kappa(\lambda)}s_\lambda(1,q,\ldots,q^{k-1}))$
exhibits the CSP when the length of $\lambda$ is less than $k$ and $\mbox{gcd}(k,\vert \lambda \vert)=1$.  The result was extended to skew shapes in \cite{alexander}.

\bigskip

\noindent (5)  In \cite{linusson}, the authors gave a CSP for semistandard tableaux of stretched hook shape $\lambda=((a+1)n,n^b)$  and rectangular content  $\mu=(n^{a+b+1})$.  They proved that 
$(SSYT((a+1)n,n^b),\mu), \langle \partial \rangle, f(q))$ exhibits the CSP, where $$f(q)=\prod_{1 \leq i \leq a}\prod_{1 \leq j \leq b} \frac{[i+j+n-1]_q}{[i+j-1]_q}=q^{-n\binom{b+1}{2}}\widetilde{K}_{\lambda,\mu}(q).$$ Here $\widetilde{K}_{\lambda,\mu}(q)$ is a modified Kostka-Foulkes polynomial.

In this paper, we give a CSP for the set of semistandard tableaux $SSYT(\lambda,\mu)$ of shape $\lambda=(m,n^b)$ and content $\mu=(\mu_1,\ldots,\mu_{b+2})$, where $m, n, b$ are positive integers.   The shape is a more general version of the stretched hook shape $\lambda=((a+1)n,n^b)$ in (5) and our content is a $(b+2)$-tuple whereas the content in (5) is rectangular of the form $(n^{a+b+1})$.  The CSP polynomial is a $q$-binomial coefficient, which is a modified Kostka-Foulkes polynomial.  Our CSP coincides with (5) in the case where $a=1$; that is when $\lambda=(2n,n^b)$ and $\mu=(n^{b+2})$.  

After reviewing the necessary definitions and results in Sections \ref{sec:prelim1} and \ref{sec:prelim2} we prove our main result in Section \ref{sec:main}, which is that  $(SSYT(\lambda,\mu), \langle \partial^{b+2} \rangle, q^{-n\binom{b+1}{2}}\widetilde{K}_{\lambda,\mu}(q))$ is a CSP, for $\lambda=(m,n^b)$, $\mu=(\mu_1,\ldots,\mu_{b+2})$, and  $\widetilde{K}_{\lambda,\mu}(q)$ a modified Kostka-Foulkes polynomial.

\section{Semistandard tableaux and jeu de taquin promotion}\label{sec:prelim1}

A weakly decreasing $r$-tuple $\lambda=(\lambda_1,\ldots,\lambda_r)$ is a partition of a positive integer $n$ if $\lambda_i \geq 0$ and $\sum_{i=1}^r \lambda_i=n$.  The Young diagram of shape $\lambda$ consists of $n$ boxes in $r$ left-justified rows with the $i$th row containing $\lambda_i$ boxes.  A $\lambda$-tableau $T$ is obtained by filling the Young diagram with positive integers.  A $\lambda$-tableau is {\em semistandard} if the entries in its columns are strictly increasing from top to bottom and the entries in its rows are weakly increasing from left to right.  If $T$ contains entries from the set $\{1, \ldots, k\}$, the {\em content} of $T$ is the $k$-tuple $\mu=(\mu_1,\ldots,\mu_k)$ where $\mu_i$ is equal to the number of entries equal to $i$ in $T$.  We will denote the set of semistandard $\lambda$-tableau with content $\mu$ by $SSYT(\lambda, \mu)$.  

\begin{exa} \label{firstexa} The tableau $T=\begin{ytableau}1& 1 & 2 & 3 & 5 \cr 2 & 3 & 4 \cr 3 \cr \end{ytableau}$ belongs to $SSYT(\lambda,\mu)$ where $\lambda=(5,3,1)$ and $\mu=(2,2,3,1,1).$

\end{exa}

{\em Jeu-de-taquin promotion} (\cite{schutz1,schutz2}) is a combinatorial algorithm that gives an action on semistandard tableaux.  We will use the version defined in \cite{bms}, which is the inverse of the operation used in \cite{linusson}.  For a semistandard tableau $T$ with entries in $\{1,\ldots,k\}$, first  replace each entry equal to $k$ with a dot.  If there is a dot in the figure that is not contained in a continuous strip of dots in the northwest corner, choose the westernmost dot and slide it north or west until it lands in a connected component of dots in the northwest corner according to the following rules:  
$$ 
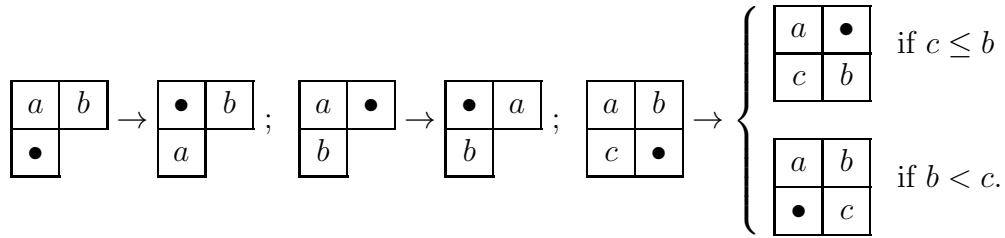
\begin{ytableau} a & b \cr  \bullet \cr \end{ytableau} \rightarrow \begin{ytableau} \bullet & b \cr a \cr \end{ytableau}\ ; \ \ \begin{ytableau} a & \bullet \cr b   \cr\end{ytableau} \rightarrow \begin{ytableau} \bullet & a \cr b \cr \end{ytableau}\ ; \ \ \begin{ytableau} a & b \cr c & \bullet \cr \end{ytableau} \rightarrow
\left\{\begin{array}{ll}
\begin{ytableau} a & \bullet \cr c & b \cr\end{ytableau} & \text{if $c \leq b$} \\ \\
\begin{ytableau} a & b \cr \bullet & c \cr \end{ytableau} & \text{if $b<c$.} 
\end{array}\right. $$

Repeat for the remaining dots, then replace each dot with $1$ and increase all other entries by one, giving $\partial(T)$, which is semistandard.  If $T$ has content $\mu=(\mu_1,\ldots,\mu_ k)$, then $\partial(T)$ has content $(\mu_k,\mu_1,\ldots,\mu_{k-1})$.

\begin{exa} Below is an illustration of jeu-de-taquin promotion.
$$T=
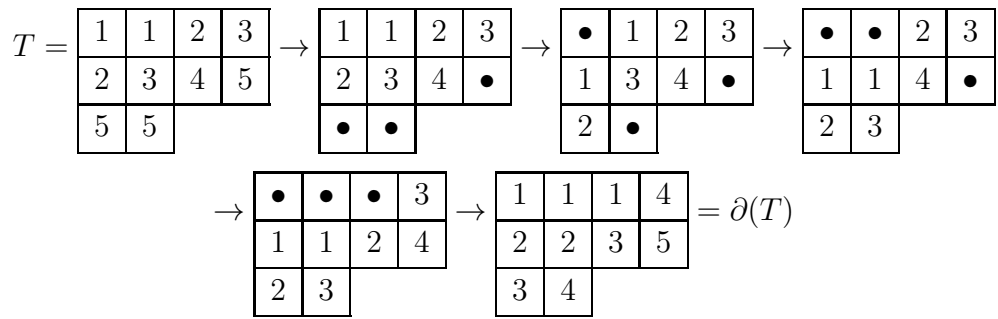
\begin{ytableau} 1 & 1 & 2 & 3 \cr 2 & 3 & 4 & 5 \cr 5 & 5 \cr\end{ytableau} \rightarrow \begin{ytableau} 1 & 1 & 2 & 3 \cr 2 & 3 & 4 & \bullet \cr \bullet & \bullet \cr \end{ytableau} \rightarrow \begin{ytableau} \bullet & 1 & 2 & 3 \cr 1 & 3 & 4 & \bullet \cr 2 & \bullet \cr \end{ytableau} \rightarrow \begin{ytableau}\bullet & \bullet & 2 & 3 \cr 1 & 1 & 4 & \bullet \cr 2 & 3 \cr \end{ytableau}$$
$$ \rightarrow \begin{ytableau} \bullet & \bullet & \bullet & 3 \cr 1 & 1 & 2 & 4 \cr 2 & 3 \cr \end{ytableau} \rightarrow \begin{ytableau} 1 & 1 & 1 & 4 \cr 2 & 2 & 3 & 5 \cr 3 & 4 \cr \end{ytableau} =\partial(T)$$

\end{exa}

The {\em order of promotion} of a tableau $T$ is the least positive integer $r$ such that $\partial^r(T)=T$.  If a set $X$ of semistandard tableaux is invariant under $\partial$, the least positive integer $r$ such that $\partial^r(T)=T$ for all $T \in X$ is the {\em order of promotion} on $X$.

\section{Kostka-Foulkes polynomials}\label{sec:prelim2}
The \textit{Kostka-Foulkes polynomials}, denoted $K_{\lambda, \mu}(q)$, relate Hall-Littlewood polynomials to Schur polynomials (see \cite{butler} for a comprehensive overview).  They generalize the Kostka coefficients $K_{\lambda \mu}$, since $K_{\lambda,\mu}(1)=K_{\lambda \mu}$, which is the number of semistandard tableaux of shape $\lambda$ and content $\mu$.  It was shown by Lascoux and Sch\"utzenberger \cite{lascoux} that the Kostka-Foulkes polynomials can be found using a statistic, called {\em charge}, which had previously been conjectured by Foulkes \cite{foulkes}:
$$K_{\lambda,\mu}(q)=\sum_{T \in SSYT(\lambda,\mu)}q^{charge(T)}.$$

We will work with \textit{modified Kostka-Foulkes polynomials} $\widetilde{K}_{\lambda,\mu}(q)$, which are related to the Kostka-Foulkes polynomials by the relation  $\widetilde{K}_{\lambda,\mu}(q)=q^{\kappa(\mu)}K_{\lambda,\mu}(q^{-1})$, where $\kappa(\mu)=\sum_i(i-1)\mu_i$.  These can be obtained via a statistic on tableaux called {\em cocharge}, denoted $cc(T)$, which we will define shortly: $$\widetilde{K}_{\lambda, \mu}(q)=\sum_{T \in SSYT(\lambda,\mu)}q^{cc(T)}.$$

Given a permutation $w=w_1\ldots w_n \in \mathfrak{S}_n$, where $\mathfrak{S}_n$ is the symmetric group on $n$ letters, define the {\em cocharge} of $j$ in $w$ recursively as follows:
\begin{equation*}
 cc(w,j) := \left\{
        \begin{array}{ll}
            0 & \quad \text{if} \,  j=1 \\
           cc(w,j-1)+1 & \quad \text{if $j$ precedes $j-1$ in} \,  w \\
            cc(w,j-1)  & \quad \text{otherwise}.
        \end{array}
    \right.
\end{equation*}
The cocharge of the word $w$ is  $cc(w)=\sum_{j=1}^n cc(w,j)$ and $charge(w)=\binom{n}{2} - cc(w)$. 
The content of a word $w$ is $\mu=(\mu_1,\ldots,\mu_n)$, where $\mu_i$ records the number of entries $i$ in $w$.  We can define cocharge for a word $w$ whenever its content $\mu$ is a partition.   To do so, obtain $\mu_1$ standard subwords from $w$ in the following way: start by selecting the rightmost $1$, then move left to find the rightmost $2$ that precedes the chosen $1$ and if there is not a $2$ preceding the $1$, loop around to the beginning of the word to choose the rightmost $2$.  Continue for $3,4,$ etc., until the largest entry in the word has been selected.  The selected entries, listed in the order they appear in $w$, form the first standard subword $w^{(1)}$.  Delete the entries in $w^{(1)}$ from $w$ and repeat the process with the word consisting of the remaining entries to obtain $w^{(2)}$.  Continue until no entries in the word remain, forming $\mu_1$ subwords.  Each of the subwords $w^{(i)}$ is a permutation, and $w^{(1)}, \ldots,w^{(k)}$ are the parts of the conjugate partition $\mu^t$.  Define the cocharge of $w$ as $cc(w)=\sum_{i=1}^{\mu_1} cc(w^{(i)})$.  

To get the cocharge of a tableau, we work with its reading word  $\mbox{rw}(T)$, which is obtained by listing the entries of $T$, left to right, across the rows, starting with the bottom row of $T$.  Define
 $cc(T)=cc(\rw(T))$.  For the cocharge of $T$ to be well-defined, it is necessary for the content of $T$ to be a partition.  However, $\widetilde{K}_{\lambda,\mu}=\widetilde{K}_{\lambda,\sigma \mu}$, where $\sigma$ is a permutation and $\sigma (\mu_1,\ldots,\mu_n)=(\mu_{\sigma(1)}, \ldots, \mu_{\sigma(n)})$ so this does not impede the use of cocharge to find the modified Kostka-Foulkes polynomial.
 
 \begin{exa} Let $T=\begin{ytableau} 1 & 1 & 1 & 1 & 2 & 2 & 3 & 4 \cr 2 & 2 & 3 \cr 3 & 4 & 4 \cr \end{ytableau}$ with $\rw(T)=34422311112234$.  The content of $w$, which is the content of $T$, is $\mu=(4,4,3,3)$.  The four standard subwords obtained from $w$ are:
 $w^{(1)}=3214, \ w^{(2)}=4213, \ w^{(3)}=4312, \ w^{(4)}=12$.  Then $cc(w^{(1)})=cc(w^{(1)},1)+cc(w^{(1)},2)+cc(w^{(1)},3)+cc(w^{(1)},4)=0+1+2+2=5$, $cc(w^{(2)})=0+1+1+2=4, cc(w^{(3)})=0+0+1+2=3,
 cc(w^{(4)})=0+0=0$ so $cc(w)=12$. \end{exa}

\section{Main result}\label{sec:main}

Our aim in this section is to prove a CSP for semistandard tableaux of shape $\lambda=(m,n^b)$ and content $\mu=(\mu_1,\ldots,\mu_{b+2})$.  For $\lambda$ and $\mu$ so defined, let $$\displaystyle \beta(\lambda,\mu)=m-n-\sum_{i=1}^{b+2}\gamma_i, \mbox{ where }
\gamma_i= \begin{cases} 
      \mu_i-n &\mbox{if } \mu_i>n \\
      0 & \text{otherwise}.
      \end{cases} $$

If $\mu_i>n$, at least $\gamma_i=\mu_i-n$ entries equal to $i$ are forced into the last $m-n$ columns of $T\in SSYT(\lambda,\mu)$ and we will refer to these as {\em forced entries}.  Thus $\sum_{i=1}^{b+2}\gamma_i$ is the number of entries in the last $m-n$ columns that are fixed and there are $m-n-\sum_{i=1}^{b+2} \gamma_i$ boxes in the last $m-n$ columns for which the entries can vary.  The entries remaining in the last $m-n$ columns of $T$ after deleting $\gamma_i$ entries equal to $i$, for each $1 \leq i \leq b+2$, will be called the {\em free entries} in $T$. Each tableau $T \in SSYT(\lambda,\mu)$ has $\beta(\lambda,\mu)$ free entries, which belong to the set $\{2,\ldots,b+2\}$.  Furthermore, since the sum $\sum_{i=1}^{b+2}\gamma_i$  is the same for any permutation $\sigma \mu$ of the content, any tableau in $SSYT(\lambda,\sigma \mu)$ also has $\beta(\lambda,\mu)$ free entries.

The free entries in $T \in SSYT(\lambda,\mu)$ can also be determined by considering a multiset of elements from $\{2,\ldots,b+2\}$ that are missing from the first $n$ columns.  Each of the first $n$ columns of $T$ is necessarily missing one element from $\{1,\ldots,b+2\}$ and the collection of these elements forms a multiset.  If $\mu_i <n$, then $i$ is missing from at least $n-\mu_i$ of the first $n$ columns in any tableau, so for each $i$ in the multiset with $\mu_i<n$, remove $n-\mu_i$ entries equal to $i$   to get a multiset $A_T$ of elements from $\{2,\ldots,b+2\}$.    The set $A_T$ consists precisely of the free elements in $T$ so $\beta(\lambda,\mu)=\vert A_T \vert=n-\sum_{\mu_i<n} (n-\mu_i)$.  Let $\mathcal{A}$ denote the set of $\beta(\lambda,\mu)$-element multisets of $\{2,\ldots,b+2\}$ and define a map $$\phi:SSYT(\lambda,\mu) \rightarrow \mathcal{A} \mbox{ where } \phi(T)=A_T.$$

Since $\vert A_T \vert \leq n$, for $T \in SSYT(\lambda,\mu)$, the following lemma is immediate.

\begin{lem}\label{betalem}
Suppose that $\lambda=(m,n^b)$ and  $\mu=(\mu_1,\ldots,\mu_{b+2})$.  Then $\beta(\lambda,\mu) \leq n$.
\end{lem}

\begin{exa} Let $T  =\ \   \begin{ytableau}
    1 & 1 & 1 & 1 & 1 & {\bf 1} & 2 & 4 & 4& {\bf 4} & {\bf 4} & {\bf 6}  \\
    2 & 2 & 2 & 3 & 3 \\
    3 & 3 & 4 & 4 & 4  \\
    5 & 5 & 5 & 5 & 5 \\
    6 & 6 & 6 & 6 & 6 \\
    \end{ytableau}\ ,$

\noindent where $\lambda=(12,5^4)$ and $\mu=(6,4,4,7,5,6)$.  Then $\gamma_1=\gamma_6=1,\gamma_4=2, \gamma_2=\gamma_3=\gamma_5=0$ and $\beta(\lambda,\mu)=3$. Entries corresponding to $\gamma_1,\gamma_4,\gamma_6$, which are forced into the top row, are boldfaced in the tableau.  The missing elements from the first five columns of $T$ are $\{4,4,3,2,2\}$.  Since $\mu_2, \mu_3<5$ and $n-\mu_2=n-\mu_3=1$, we remove both a $3$ and a $2$ to get $\phi(T)=A_T=\{4,4,2\}$. These are the free entries in the arm of the first row, which are not boldfaced.  \end{exa}

\noindent We will use the following straightforward fact in the proofs that follow.
\begin{lem}  Suppose that $T$ is a semistandard tableau of shape $\lambda=(m,n^b)$ and content $\mu=(\mu_1,\ldots,\mu_{b+2})$.  Then any row $i$ of $T$, where $i \geq 2$, can contain only the entries $i$ or $i+1$. 
\end{lem}

\begin{lem} \label{lem:sigma} Suppose that $\lambda=(m,n^b)$, $\mu=(\mu_1,\ldots,\mu_{b+2})$ and that $T \in SSYT(\lambda,\mu)$. Let $\sigma=(2,3\ldots,b+2) \in \mathfrak{S}_{b+2}$.  Then $\phi(\partial(T))=\sigma \phi(T)$. \end{lem}
\begin{proof}
Since jeu de taquin promotion permutes the content of $T$, $\vert \phi(\partial(T)) \vert=\vert \phi(T)\vert=\beta(\lambda,\mu)$.  Let $f_i^T$ denote the number of $i$'s in the multiset $\phi(T)=A_T$ and $c_i^T$  the number of $i$'s in the first $n$ columns of $T$. If $\mu_i>n$ then $f_i^T=n- c_i^T$, and if $\mu_i<n$ then $f_i^T=n-c_i^T-(n-\mu_i)=\mu_i-c_i^T$.    Any entry $i$ in $T$ with $3 \leq i \leq b+1$ either belongs to the first $n$ columns below the first row or in the last $m-n$ columns and after jeu de taquin promotion becomes an $i+1$ that belongs to the first $n$ columns below the first row of $\partial(T)$ or in the last $m-n$ columns of $\partial(T)$, respectively, so for  $3 \leq i \leq b+1$, $c_i^T=c_{i+1}^{\partial(T)}$, which yields $f_i^T=f_{i+1}^{\partial(T)}$.   

Entries equal to $2$ belong to either the first or second row of the tableau.  Those below the first row or in the last $m-n$ columns move to boxes below the first row or in the last $m-n$ columns and become $3$'s under jeu de taquin promotion.  Any of the first $n$ columns that contains a $2$ in the first row contains a $b+2$ in the last row.   If there are also $(b+2)$'s in row $b+1$ with $1$'s above them in the first row, these are moved first by jeu de taquin promotion.  Since the top entry in the column  in this case is equal to $1$,  then for some $i$ the $i$th row contains the entry $i$ while the row beneath it contains the entry $i+2$.  Jeu de taquin promotion then commences in the following way, beginning with the leftmost column that contains a $b+2$ in row $b+1$:

$$\begin{ytableau}  1 & 1 \cr  2 & 2 \cr  \vdots & \vdots \cr i & i \cr  \scriptstyle{i+1} & \scriptstyle{i+2} \cr \vdots  & \vdots \cr \sc{b+1} & \bullet \end{ytableau} \rightarrow \begin{ytableau} 1 & 1 \cr 2 & 2 \cr \vdots & \vdots \cr i & i \cr \sc{i+1} & \bullet \cr \vdots & \vdots \cr \sc{b+1} & \sc{b+1} \cr \end{ytableau}
\rightarrow \begin{ytableau} 1 & 1 \cr 2 & 2 \cr \vdots & \vdots \cr i & i \cr \bullet & \sc{i+1} \cr \vdots & \vdots \cr \sc{b+1} & \sc{b+1} \cr \end{ytableau}$$

The jeu de taquin promotion path then moves left across  row $i+1$ to the first column without a dot and north to the first row, replacing a $1$ with a dot. Promotion behaves in the same way for all remaining columns that contain a $b+2$ in the last row and a $1$ in the top row.  For columns that contain $(b+2)$'s in the last row and entries equal to $2$ in the first row, it is now the case that for the leftmost such column, any entry $i$ in the column has an $i-1$ immediately to its left.  Thus, jeu de taquin promotion slides the box in the last row directly to row one, which moves the $2$ into the second row.  Each remaining $b+2$ in row $b+1$ behaves in the same way, sliding each remaining $2$ into the second row.  It follows that $c_2^T=c_3^{\partial(T)}$ so $f_i^T=f_{i+1}^{\partial(T)}$ for $2 \leq i \leq b+1$.

Since $\displaystyle \beta(\lambda,\mu)=\sum_{i=2}^{b+2} f_i^{\partial(T)}=f_2^{\partial(T)}+\sum_{i=2}^{b+1}f_i^T$ and $\displaystyle \beta(\lambda,\mu)=\sum_{i=2}^{b+2}f_i^T=f_{b+2}^T+\sum_{i=2}^{b+1}f_i^T$, we have $f_2^{\partial(T)}=f_{b+2}^T$ and the result follows. \end{proof}

The following lemma shows that if $\lambda=(m,n^b)$ and $T \in SSYT(\lambda,\mu)$  has fixed content $\mu=(\mu_1,\ldots,\mu_{b+2})$, then $T$ is uniquely determined by its free entries.  It follows that the map $\phi:SSYT(\lambda,\mu) \rightarrow \mathcal{A}$ is a bijection.

\begin{lem}\label{lem:unique} Let $\lambda=(m,n^b)$, $\mu=(\mu_1,\ldots,\mu_{b+2})$ and let $T \in SSYT(\lambda,\mu)$.  Then $T$ is uniquely determined by the multiset $A_T$. \end{lem}
\begin{proof}
For each $i$  with $\mu_i<n$, add $n-\mu_i$ elements $i$ to  $A_T$ to produce an $n$-element multiset $X$.  Since $T$ is semistandard, listing the elements of $X$ in weakly decreasing order completely determines the entries in $\{1,\ldots,b+2\}$ that are missing from each of the  first $n$ columns of $T$.  The complement of the $k$th element in the multiset gives the $k$th column of $T$.  The remaining entries in $T$, determined from $\mu$, appear in weakly increasing order in the first row of $T$.
\end{proof}

Since jeu de taquin promotion permutes the content of a tableau, the content of $T$ and $\partial^{b+2}(T)$ are equal so $\partial^{b+2}:SSYT(\lambda,\mu) \rightarrow SSYT(\lambda,\mu)$ for $\lambda=(m,n^b)$ and $\mu=(\mu_1,\ldots,\mu_{b+2})$.    By Lemma \ref{lem:sigma}, if $T \in SSYT(\lambda,\mu)$ and $\sigma=(2,\ldots,b+2)$ then $\phi(\partial^{(b+2)j}(T))=\sigma^{j}\phi(T)$.  Thus $\phi({\partial^{(b+2)(b+1)}(T)})=\sigma^{b+1}\phi(T)=\phi(T)$ so $\partial^{(b+2)(b+1)}(T)=T$.  If $1\leq j <b+1$ satisfies $\partial^{(b+2)j}(T)=T$ for all $T$, then $\sigma^j\phi(T)=\phi(T)$ for all $T$, which is not possible, so have the following Lemma.

\begin{lem} Let $\lambda=(m,n^b)$ and $\mu=(\mu_1,\ldots,\mu_{b+2})$.  The {\em order of promotion} on $SSYT(\lambda,\mu)$ under the cyclic action of $\partial^{b+2}$ is equal to $b+1$. \end{lem}

For a positive integer $n$, let
$\ds [n]_q=\frac{q^{n}-1}{q-1}$ and 
$[n]_q!=[n]_q[n-1]_q \cdots [1]_q$.  The {\em $q$-binomial
coefficients} are defined by
$\ds \left[\begin{array}{c}n \cr k
\cr\end{array}\right]_q=\frac{[n]_q!}{[k]_q![n-k]_q!}.$
 To prove our CSP, we will use the bijection $\phi$  and invoke the following theorem due to Reiner, Stanton and White.

\begin{thm}\label{thm:rsw} (Reiner, Stanton and White \cite{rsw}) Let $X$ be the set of $k$-element multisets of $\{1,2,\ldots,n\}$, let $C=\mathbb{Z}/n\mathbb{Z}$ act on $X$ via the permutation $\theta=(1,2,\ldots,n)$ and let $\displaystyle f(q)=\left[ \begin{array}{c} n+k-1 \cr k \cr \end{array} \right]_q$.  Then $(X,C,f(q))$ exhibits the cyclic sieving phenomenon. \end{thm}

\begin{thm}\label{thm:main} Let $\lambda=(m,n^b)$, $\mu=(\mu_1,\ldots,\mu_{b+2})$, let $C=\mathbb{Z}/(b+1)\mathbb{Z}$ act on $SSYT(\lambda,\mu)$ via $\partial^{b+2}$ and let $f(q)=\left[
\begin{array}{c} b+\beta(\lambda,\mu) \cr \beta(\lambda,\mu) \cr \end{array} \right]_q$ .  Then $(SSYT(\lambda,\mu),C,f(q))$ exhibits the cyclic sieving phenomenon.\end{thm}
\begin{proof} Adjust the map $\phi$ by decrementing each of the entries in $\phi(T)$ to get $\psi:SSYT(\lambda,\mu) \rightarrow \mathcal{B}$, where $\mathcal{B}$ is the set of $\beta(\lambda,\mu)$-element multisets of $\{1,\ldots,b+1\}$.  Let $\theta=(1,2,\ldots,n)$.  By Lemma \ref{lem:sigma}, $\psi(\partial^{b+2}(T))=\theta(\psi(T))$ and $\psi(\partial^{(b+2)j}(T))=\theta^j(\psi(T))$. 

We have $\psi^{(b+2)j}(T)=T$ if and only if $\psi(\partial^{(b+2)j}(T))=\psi(T)$, which, by the above, yields $\theta^j(\psi(T))=\psi(T)$.  But, by Theorem \ref{thm:rsw}, $\vert \mathcal{B}^{\theta^j} \vert=f(\omega^j)$, where $\omega$ is a
 primitive $(b+1)$-th root of unity.  The result now follows.  \end{proof}

\bigskip

We will now examine the relationship between the cyclic sieving polynomial in Theorem \ref{thm:main} and the modified Kostka-Foulkes polynomial $\widetilde{K}_{\lambda,\mu}$.  To do so, we work with plane partitions to get a nice formula for cocharge in the case where $\lambda=(m,n^b)$ and $\mu=(\mu_1,\ldots,\mu_{b+2})$.  

A \textit{plane partition} is an array $\pi = (\pi_{ij})_{i,j \geq 1}$ of nonnegative integers such that $\pi$ has finitely many nonzero entries and is weakly decreasing in rows and columns. If $\sum \pi_{ij} = n$, we write $\vert{\pi}\vert=n$ and say that $\pi$ is a plane partition of $n$.  We can adjust the bijection $\phi$ between $SSYT(\lambda,\mu)$ and the set of $\beta(\lambda,\mu)$-element multisets of $\{2,\ldots,b+2\}$ by subtracting two from each of the entries in $\phi(T)$  and reversing the order to get a bijection between $SSYT(\lambda,\mu)$ and the set of one-row plane partitions $\pi=(\pi_1,\ldots,\pi_{\beta(\lambda,\mu)})$ with $\pi_1\leq b$ and $\beta(\lambda,\mu)$ columns; we will denote the image of $T$ under this bijection by $\pi_T$. 

\bigskip

\begin{thm}\label{thm:cc} Let $\lambda=(m,n^b)$, let $\mu=(\mu_1,\ldots,\mu_{b+2})$ be a partition of $m+nb$ and let $T \in SSYT(\lambda,\mu)$.  Then  $\displaystyle \cc(\rw(T))=\vert \pi_T \vert + n \binom{b+1}{2}$, where
$\vert \pi_T \vert$ is the sum of the entries in $\pi_T$. \end{thm}

\begin{proof}  
We will consider the contribution of a given entry $i$ in the tableau to the cocharge.   An entry $i$ copies the cocharge contribution of $i-1$ in its subword if $i-1$ precedes $i$ in its subword (an $(i-1,i)$ pairing) and it increases cocharge otherwise (an $(i,i-1)$ pairing).  

If an entry $i$ belongs to row $i$, each of the entries $1, \ldots, i-1$ appear above it in the same column so $i$ pairs with an $i-1$ in a row above it, giving associated subword $w=\cdots i (i-1) \cdots 3 2 1$.  Thus, every entry $i$ in row $i$ contributes $i-1$ to the cocharge.  

Any entry $i$ in row one, where $i \geq 2$, pairs as $(i-1,i)$ in the associated subword so copies the cocharge of $i-1$.  Any subword containing a forced entry consists entirely of forced entries.  Since there are $\mu_1-n$ forced $1$'s, there are $\mu_1-n$ subwords consisting of all forced entries, and these are of the form $w=12\cdots i$, so each forced entry contributes zero to the cocharge.  

If $i-1$ belongs to one of the first $n$ columns, it either belongs to row $i-1$ or to row $i-2$.  If $i$ is a free entry in row one, it cannot pair with a forced entry $i-1$ and, if it pairs with an entry $i-1$ in row $i-2$ or with a free entry $i-1$ in row one, this forces an entry $i$ in row $i-1$. It would then follow that $i-1$ pairs with an entry $i$ in the row beneath it, instead of the $i$ in the first row.  Thus each free entry $i$ in row one pairs with an $i-1$ in row $i-1$, creating an $(i-1,i)$ pairing in the subword so copies the contribution of $i-1$ to the cocharge.  Thus each free entry $i$ in row one contributes $i-2$ to the cocharge.

Finally, we will show that entries $i+1$ in row $i$ contribute $i-1$ to the cocharge.  If $i+1$ pairs with an $i$ in the same row, this yields an $(i,i+1)$ pairing in the subword so $i+1$ contributes the same value to cocharge as $i$, which is $i-1$.  If $i+1$ pairs with a free entry $i$ in row one this yields an $(i+1,i)$ pairing in the subword so that $i+1$ increases the contribution of the free entry, giving a contribution of $i-1$.  The last case is when $i+1$ pairs with an $i$ in row $i-1$ creating an $(i,i-1)$ pairing in the subword so increasing the contribution of $i-1$ by one.  By induction, $i-1$ contributes $i-2$ to the cocharge so $i+1$ contributes $i-1$ to cocharge.  It follows that all entries in row $i$ contribute $i-1$ to the cocharge so
$\ds cc(\rw(T))=\vert \pi_T \vert + n \sum_{i=1}^{b+1} (i-1)=\vert \pi_T \vert + n \binom{b+1}{2}.$ \end{proof}

\begin{cor} Let $\lambda=(m,n^b)$, $\mu=(\mu_1,\ldots,\mu_{b+2})$ and $f(q)=\left[
\begin{array}{c} b+\beta(\lambda,\mu) \cr \beta(\lambda,\mu) \cr \end{array} \right]_q$.  Then $f(q)=q^{-n\binom{b+1}{2}}\widetilde{K}_{\lambda,\mu}(q).$  \end{cor}

\begin{proof}
Denoting $SSYT(\lambda,\mu)$ by $SSYT$, the result follows from Lemma \ref{thm:cc}  since
\begin{eqnarray*}\widetilde{K}_{\lambda,\mu}(q)=\sum_{T\in SSYT} q^{cc(rw(T))}&=&q^{n \binom{b+1}{2}}\sum_{T\in SSYT} q^{\vert \pi_T \vert}\cr
&=&q^{n \binom{b+1}{2}}\sum_{\substack{\pi=(\pi_1,\ldots,\pi_{\beta(\lambda,\mu)}) \\ \pi_1 \leq b}}q^{\vert \pi \vert} = q^{n \binom{b+1}{2}}\left[
\begin{array}{c} b+\beta(\lambda,\mu) \cr \beta(\lambda,\mu) \cr \end{array} \right]_q,\cr \end{eqnarray*}  by \cite[I.3.19]{stanley1} (see also \cite[\S 7.21]{stanleybook}).
\end{proof}
\noindent {\bf Acknowledgement.}  The authors wish to thank two anonymous referees who provided useful suggestions that improved the paper, including an equivalent variation on our map between tableaux and multisets.

\end{document}